\documentclass[reqno]{amsart}
\usepackage{amssymb}
\usepackage{amsfonts}
\usepackage{amsmath}

\newtheorem{theorem}{Theorem}
\theoremstyle{plain}
\newtheorem{acknowledgement}{Acknowledgement}

\newtheorem{lemma}{Lemma}

\newtheorem{remark}{Remark}

\numberwithin{equation}{section}

\input{tcilatex}

\begin{document}
\author{}
\title{}
\maketitle

\begin{center}
\thispagestyle{empty} \pagestyle{myheadings} 
\markboth{\bf Yilmaz Simsek
}{\bf Combinatorial applications of the special numbers and polynomials}

\textbf{\Large Combinatorial applications of the special numbers and
polynomials}

\bigskip

\textbf{Yilmaz Simsek}\\[0pt]

\textit{Department of Mathematics, Faculty of Science University of Akdeniz
TR-07058 Antalya, Turkey,}

E-mail: ysimsek@akdeniz.edu.tr\\[0pt]

\bigskip

\textbf{{\large {Abstract}}}\medskip
\end{center}

\begin{quotation}
In this paper, by using some families of special numbers and polynomials
with their generating functions, we give various properties of these numbers
and polynomials. These numbers are related to the well-known numbers and
polynomials, which are the Euler numbers, the Stirling numbers of the second
kind, the central factorial numbers and the array polynomials. We also
discuss some combinatorial interpretations of these numbers related to the
rook polynomials and numbers. Furthermore, we give computation formulas for
these numbers and polynomials.

\bigskip
\end{quotation}

\noindent \textbf{2010 Mathematics Subject Classification.} 11B68; 05A15;
05A19; 12D10; 26C05; 30C15.

\noindent \textbf{Key Words.} Euler numbers; Central factorial numbers;
Array polynomials; Stirling numbers; Generating functions; \ Binomial
coefficients; Combinatorial sum.

\section{Introduction}

The special numbers and their generating functions have many application in
Combinatorial Number System and in Probability Theory. There are many
advantage of the generating functions. By using generating functions for
special numbers and polynomials, one can get not only various properties of
these numbers and polynomials, but also enumerating arguments such as
counting the number of subsets and the number of total ordering. In this
paper, by using generating functions and their functional equations, we
derive some identities and relations for the special combinatorial numbers
such as the Stirling numbers of the first kind, the central factorial
numbers, the Euler numbers, the array polynomials and the other special
numbers. In order to give our results, we need some special numbers and
polynomials with their generating functions.

The first kind Apostol-Euler polynomials of order $k$ are defined by means
of the following generating function:%
\begin{equation}
F_{P1}(t,x;k,\lambda )=\left( \frac{2}{\lambda e^{t}+1}\right)
^{k}e^{tx}=\sum_{n=0}^{\infty }E_{n}^{(k)}(x;\lambda )\frac{t^{n}}{n!},
\label{Cad3}
\end{equation}%
($\left\vert t\right\vert <\pi $ when $\lambda =1$ and $\left\vert
t\right\vert <\left\vert \ln \left( -\lambda \right) \right\vert $ when $%
\lambda \neq 1$), $\lambda \in \mathbb{C}$, the set of complex numbers, $%
k\in \mathbb{N}$, the set of natural numbers. By (\ref{Cad3}), we easily see
that%
\begin{equation*}
E_{n}^{(k)}(\lambda )=E_{n}^{(k)}(0;\lambda ),
\end{equation*}%
which denotes the first kind Apostol-Euler numbers of order $k$. By
substituting $k=\lambda =1$ into (\ref{Cad3}), we have%
\begin{equation*}
E_{n}=E_{n}^{(1)}(1)
\end{equation*}%
which denotes the first kind Euler numbers (\textit{cf}. \cite{Bayad}-\cite%
{SrivastavaLiu}, and the references cited therein).

The second kind Euler numbers $E_{n}^{\ast }$ of negative order are defined
by means of the following generating function:%
\begin{equation}
F_{E2}(t,k)=\left( \frac{2}{e^{t}+e^{-t}}\right) ^{-k}=\sum_{n=0}^{\infty
}E_{n}^{\ast (-k)}\frac{t^{n}}{n!},  \label{E2K}
\end{equation}%
where $\left\vert t\right\vert <\frac{\pi }{2}$ (\textit{cf}. \cite{Byrd}, 
\cite{SimsekNEW}, and the references cited therein).

The $\lambda $-Stirling numbers of the second kind $S_{2}(n,v;\lambda )$
defined by means of the following generating function:%
\begin{equation}
F_{S}(t,v;\lambda )=\frac{\left( \lambda e^{t}-1\right) ^{v}}{v!}%
=\sum_{n=0}^{\infty }S_{2}(n,v;\lambda )\frac{t^{n}}{n!},  \label{SN-1}
\end{equation}%
where $v\in \mathbb{N}_{0}$ and $\lambda \in \mathbb{C}$ (\textit{cf}. \cite%
{Luo}, \cite{SimsekFPTA}, \cite{Srivastava2011}, and the references cited
therein).

By using (\ref{SN-1}), we have%
\begin{equation*}
S_{2}(n,v)=\frac{1}{v!}\sum_{j=0}^{v}\left( 
\begin{array}{c}
v \\ 
j%
\end{array}%
\right) (-1)^{v-j}\lambda ^{j}j^{n}
\end{equation*}%
(\textit{cf}. \cite{Luo}, \cite{SimsekFPTA}, \cite{Srivastava2011}).

Substituting $\lambda =1$ into (\ref{SN-1}), we have the Stirling numbers of
the second kind $S_{2}(n,v)$ which denotes the number of ways to partition a
set of $n$ objects into $v$ groups:%
\begin{equation*}
S_{2}(n,v)=S_{2}(n,v;1).
\end{equation*}%
(\textit{cf}. \cite{Alayont}-\cite{SrivastavaLiu}; see also the references
cited in each of these earlier works).

In \cite{SimsekFPTA}, we defined the $\lambda $-array polynomials $%
S_{v}^{n}(x;\lambda )$ by means of the following generating function:%
\begin{equation}
F_{A}(t,x,v;\lambda )=\frac{\left( \lambda e^{t}-1\right) ^{v}}{v!}%
e^{tx}=\sum_{n=0}^{\infty }S_{v}^{n}(x;\lambda )\frac{t^{n}}{n!},
\label{ARY-1}
\end{equation}%
where $v\in \mathbb{N}_{0}$ and $\lambda \in \mathbb{C}$ (\textit{cf}. \cite%
{Gradimir}, \cite{Bayad}, \cite{Chan}, \cite{SimsekFPTA}, \cite{AM2014}, and
the references cited therein).

The central factorial numbers $T(n,k)$ (of the second kind) are defined by
means of the following generating function:%
\begin{equation}
F_{T}(t,k)=\frac{1}{(2k)!}\left( e^{t}+e^{-t}-2\right)
^{k}=\sum_{n=0}^{\infty }T(n,k)\frac{t^{2n}}{(2n)!}  \label{CT-1}
\end{equation}%
(\textit{cf}. \cite{Bona}, \cite{Cigler}, \cite{Comtet}, \cite{AM2014}, \cite%
{SrivastavaLiu}, and the references cited therein).

\begin{remark}
	The central factorial numbers are used in combinatorial problems. That is
	the number of ways to place $k$\ rooks on a size $m$\ triangle board in
	three dimensions is equal to 
	\begin{equation*}
	T(m+1,m+1-k)\mathit{,\ }
	\end{equation*}%
	where $0\leq k\leq m$ (\textit{cf}. \cite{Alayont}).
\end{remark}

In \cite{SimsekNEW}, we defined the numbers $y_{1}(n,k;\lambda )$ by means
of the following generating functions:%
\begin{equation}
F_{y_{1}}(t,k;\lambda )=\frac{1}{k!}\left( \lambda e^{t}+1\right)
^{k}=\sum_{n=0}^{\infty }y_{1}(n,k;\lambda )\frac{t^{n}}{n!},  \label{ay1}
\end{equation}%
where $k\in \mathbb{N}_{0}$ and $\lambda \in \mathbb{C}$. If we substitute $%
\lambda =-1$ into (\ref{ay1}), then we get the Stirling numbers of the
second kind, $S_{2}(n,k)$:%
\begin{equation*}
S_{2}(n,k)=(-1)^{k}y_{1}(n,k;-1)
\end{equation*}%
(\textit{cf}. \cite{SimsekNEW}, \cite{mmas2016}). The numbers $%
y_{1}(n,k;\lambda )$ is related to following novel combinatorial sum:%
\begin{equation}
B(n,k)=k!y_{1}(n,k;1)=\sum_{j=0}^{k}\left( 
\begin{array}{c}
k \\ 
j%
\end{array}%
\right) j^{n}=\frac{d^{n}}{dt^{n}}\left( e^{t}+1\right) ^{k}\left\vert
_{t=0}\right. ,  \label{Gl}
\end{equation}%
where $n=1,2,\ldots $(cf. \cite{golombek}, \cite{SimsekNEW}). In the work of
Spivey \cite[Identity 8-Identity 10]{Spevy}, we see that%
\begin{equation*}
B(0,k)=2^{k},B(1,k)=k2^{k-1},B(2,k)=k(k+1)2^{k-2},
\end{equation*}%
and also%
\begin{equation}
B(m,n)=\sum_{j=0}^{n}\left( 
\begin{array}{c}
n \\ 
j%
\end{array}%
\right) j!2^{n-j}S_{2}(m,j)  \label{Bs-1}
\end{equation}%
(\textit{cf}. \cite[p.4, Eq-(7)]{Boyadzhiev}, \cite{SimsekNEW}; see also the
references cited in each of these earlier works). In \cite{SimsekNEW}, we a
conjecture and two open questions associated with the numbers $B(n,k)$.

In \cite{SimsekNEW}, we defined the numbers $y_{2}(n,k;\lambda )$ by means
of the following generating functions:%
\begin{equation}
F_{y_{2}}(t,k;\lambda )=\frac{1}{(2k)!}\left( \lambda e^{t}+\lambda
^{-1}e^{-t}+2\right) ^{k}=\sum_{n=0}^{\infty }y_{2}(n,k;\lambda )\frac{t^{n}%
}{n!}.  \label{C1}
\end{equation}

In \cite{SimsekNEW}, we gave some combinatorial interpretations for the
numbers $y_{1}(n,k)$, $y_{2}(n,k)$\ and $B(n,k)$ as well as the
generalization of the central factorial numbers. We see that these numbers
were related to the rook numbers and polynomials.

\section{Functional equations and related identities}

By using generating functions for the Stirling numbers, the Euler numbers, the central factorial numbers, the array polynomials, the numbers $y_{1}(n,k;\lambda )$ and the
numbers $y_{2}(n,k;\lambda )$ with their functional equations, we derive
some identities and relations involving binomial coefficients and these
numbers and polynomials. We also give computation formulas for the first
kind and the second kind Euler numbers and polynomials.

By using (\ref{ay1}) and (\ref{SN-1}), we obtain the following functional
equation:%
\begin{equation*}
F_{y_{1}}(2t,k;-\lambda ^{2})=(-1)^{k}k!F_{y_{1}}(t,k;\lambda
)F_{S}(t,k;\lambda ).
\end{equation*}%
By using the above equation, we get%
\begin{equation*}
\sum_{n=0}^{\infty }2^{n}y_{1}\left( n,k;-\lambda ^{2}\right) \frac{t^{n}}{n!%
}=(-1)^{k}k!\sum_{n=0}^{\infty }y_{1}(n,k;\lambda )\frac{t^{n}}{n!}%
\sum_{n=0}^{\infty }S_{2}\left( n,k;\lambda \right) \frac{t^{n}}{n!}.
\end{equation*}%
By using the Cauchy product in the above equation, we obtain%
\begin{equation*}
\sum_{n=0}^{\infty }2^{n}y_{1}\left( n,k;-\lambda ^{2}\right) \frac{t^{n}}{n!%
}=(-1)^{k}k!\sum_{n=0}^{\infty }\sum_{l=0}^{n}\left( 
\begin{array}{c}
n \\ 
l%
\end{array}%
\right) S_{2}(l,k;\lambda )y_{1}(n-l,k;\lambda )\frac{t^{n}}{n!}.
\end{equation*}%
Comparing the coefficients of $\frac{t^{n}}{n!}$ on both sides of the above
equation, we arrive the following theorem:

\begin{theorem}
	\begin{equation*}
	y_{1}(n,k;-\lambda ^{2})=(-1)^{k}k!2^{-n}\sum_{l=0}^{n}\left( 
	\begin{array}{c}
	n \\ 
	l%
	\end{array}%
	\right) S_{2}(l,k;\lambda )y_{1}(n-l,k;\lambda ).
	\end{equation*}
\end{theorem}

By combining (\ref{ARY-1}) with (\ref{CT-1}) and (\ref{C1}), we obtain the
following functional equation:%
\begin{equation*}
F_{A}(2t,-k,2k;1)=(2k)!F_{T}(t,k)F_{y_{2}}(t,k;1).
\end{equation*}%
Using the above equation, we get%
\begin{equation*}
\sum_{n=0}^{\infty }2^{n}S_{2k}^{n}(-k)\frac{t^{n}}{n!}=(2k)!\sum_{n=0}^{%
	\infty }T(n,k)\frac{t^{2n}}{\left( 2n\right) !}\sum_{n=0}^{\infty
}y_{2}\left( n,k;1\right) \frac{t^{2n}}{\left( 2n\right) !}.
\end{equation*}%
Therefore%
\begin{equation*}
\sum_{n=0}^{\infty }2^{n}S_{2k}^{n}(-k)\frac{t^{n}}{n!}=(2k)!\sum_{n=0}^{%
	\infty }\sum_{l=0}^{n}\left( 
\begin{array}{c}
n \\ 
l%
\end{array}%
\right) T(j,k)y_{2}\left( n-l,k;1\right) \frac{t^{2n}}{\left( 2n\right) !}.
\end{equation*}%
By using the above equation, we arrive at the following theorem:

\begin{theorem}
	\begin{equation*}
	S_{2k}^{2n}(-k)=(2k)!2^{-2n}\sum_{l=0}^{n}\left( 
	\begin{array}{c}
	n \\ 
	l%
	\end{array}%
	\right) T(l,k)y_{2}\left( n-l,k;1\right) .
	\end{equation*}
\end{theorem}

\begin{lemma}
	\label{LemmaRainville}(\cite[Lemma 11, Eq-(7)]{Rainville})%
	\begin{equation*}
	\sum_{n=0}^{\infty }\sum_{k=0}^{\infty }A(n,k)=\sum_{n=0}^{\infty
	}\sum_{k=0}^{\left[ \frac{n}{2}\right] }A(n,n-2k),
	\end{equation*}%
	where $\left[ x\right] $ denotes the greatest integer function.
\end{lemma}

By combining (\ref{ARY-1}) and (\ref{CT-1}) with (\ref{ay1}), we get the
following functional equation:%
\begin{equation*}
F_{T}(t,k)=\frac{k!}{(2k)!}\sum_{l=0}^{k}\frac{(2l)!}{l!}F_{T}\left( \frac{t%
}{2},l\right) F_{A}\left( -\frac{t}{2},\frac{l}{2},k-l;1\right) .
\end{equation*}%
By using the above equation, we obtain%
\begin{eqnarray*}
	&&\sum_{n=0}^{\infty }T(n,k)\frac{t^{2n}}{\left( 2n\right) !} \\
	&=&\frac{k!}{(2k)!}\sum_{l=0}^{k}\frac{(2l)!}{l!}\sum_{n=0}^{\infty
	}2^{-2n}T(n,l)\frac{t^{2n}}{\left( 2n\right) !}\sum_{n=0}^{\infty
}S_{k-l}^{n}\left( \frac{l}{2},1\right) \frac{t^{n}}{n!}.
\end{eqnarray*}
By using Lemma \ref{LemmaRainville}, we get%
\begin{eqnarray*}
	&&\sum_{n=0}^{\infty }T(n,k)\frac{t^{2n}}{\left( 2n\right) !} \\
	&=&\frac{k!}{(2k)!}\sum_{l=0}^{k}\frac{(2l)!}{l!}\sum_{n=0}^{\infty
	}\sum_{j=0}^{\left[ \frac{n}{2}\right] }T\left( j,l\right)
	S_{k-l}^{n-2j}\left( \frac{l}{2},1\right) \frac{2^{-2j}}{\left( 2j\right) !}%
	\frac{t^{n}}{\left( n-2j\right) !}.
\end{eqnarray*}%
Comparing the coefficients on both sides of the above equation, we arrive
the following theorem:

\begin{theorem}
	If $n$ is an even integer, we have%
	\begin{equation*}
	T(n,k)=\frac{\left( 2n\right) !k!}{(2k)!n!}\sum_{l=0}^{k}\sum_{j=0}^{\left[ 
		\frac{n}{2}\right] }\left( 
	\begin{array}{c}
	n \\ 
	2j%
	\end{array}%
	\right) \frac{(2l)!}{2^{2j}l!}T\left( j,l\right) S_{k-l}^{n-2j}\left( \frac{l%
	}{2},1\right) 
	\end{equation*}%
	and if $n$ is an odd integer, we have%
	\begin{equation*}
	\sum_{l=0}^{k}\sum_{j=0}^{\left[ \frac{n}{2}\right] }\left( 
	\begin{array}{c}
	n \\ 
	2j%
	\end{array}%
	\right) \frac{(2l)!}{2^{2j}l!}T\left( j,l\right) S_{k-l}^{n-2j}\left( \frac{l%
	}{2},1\right) =0.
	\end{equation*}
\end{theorem}

By combining (\ref{ARY-1}) with (\ref{E2K}), we obtain%
\begin{equation*}
F_{T}\left( 2t,k\right) =\frac{2^{2k}}{(2k)!}\sum_{j=0}^{k}\left( 
\begin{array}{c}
k \\ 
j%
\end{array}%
\right) (-1)^{k-j}F_{E2}(t,-2j).
\end{equation*}%
By using the above functional equation, we get%
\begin{equation*}
\sum_{n=0}^{\infty }2^{n}T(n,k)\frac{t^{2n}}{\left( 2n\right) !}=\frac{2^{2k}%
}{(2k)!}\sum_{n=0}^{\infty }\sum_{j=0}^{k}\left( 
\begin{array}{c}
k \\ 
j%
\end{array}%
\right) (-1)^{k-j}E_{n}^{\ast (-2j)}\frac{t^{2n}}{\left( 2n\right) !}.
\end{equation*}
Comparing the coefficients of $\frac{t^{2n}}{\left( 2n\right) !}$ on both
sides of the above equation, we arrive at the computation formula for the second kind Euler numbers of negative order which is given by the following theorem:

\begin{theorem}
	\begin{equation*}
	T(n,k)=\frac{2^{2k-n}}{(2k)!}\sum_{j=0}^{k}\left( 
	\begin{array}{c}
	k \\ 
	j%
	\end{array}%
	\right) (-1)^{k-j}E_{n}^{\ast (-2j)}.
	\end{equation*}
\end{theorem}

By using (\ref{C1}) and (\ref{SN-1}), we get%
\begin{equation*}
F_{y_{2}}(t,k;-\lambda )=\frac{k!}{(2k)!}\sum_{j=0}^{k}(-1)^{k}F_{S}(t,j;%
\lambda )F_{S}(-t,k-j;\lambda ^{-1}).
\end{equation*}%
By using the above functional equation, we obtain%
\begin{equation*}
\sum_{n=0}^{\infty }y_{2}(n,k;\lambda )\frac{t^{n}}{n!}=\frac{k!}{(2k)!}%
\sum_{j=0}^{k}(-1)^{k}\left( \sum_{n=0}^{\infty }S_{2}(n,j;\lambda )\frac{%
	t^{n}}{n!}\sum_{n=0}^{\infty }S_{2}\left( n,k-j;\lambda ^{-1}\right) \frac{%
	\left( -t\right) ^{n}}{n!}\right) .
\end{equation*}%
By using the Cauchy product in the right-hand side of the above equation, we
obtain%
\begin{eqnarray*}
	&&\sum_{n=0}^{\infty }y_{2}(n,k;\lambda )\frac{t^{n}}{n!} \\
	&=&\sum_{n=0}^{\infty }\frac{k!}{(2k)!}\sum_{j=0}^{k}%
	\sum_{d=0}^{n}(-1)^{k+n-d}\left( 
	\begin{array}{c}
		n \\ 
		d%
	\end{array}%
	\right) S_{2}(d,j;\lambda )S_{2}\left( n-d,k-j;\lambda ^{-1}\right) \frac{%
		t^{n}}{n!}.
\end{eqnarray*}%
Comparing the coefficients of $\frac{t^{n}}{n!}$ on both sides of the above
equation, we arrive at the following theorem:

\begin{theorem}
	\begin{equation*}
	y_{2}(n,k;\lambda )=\frac{k!}{(2k)!}\sum_{j=0}^{k}\sum_{d=0}^{n}(-1)^{k+n-d}%
	\left( 
	\begin{array}{c}
	n \\ 
	d%
	\end{array}%
	\right) S_{2}(d,j;\lambda )S_{2}\left( n-d,k-j;\lambda ^{-1}\right) .
	\end{equation*}
\end{theorem}

Computation formula for the first kind Euler polynomials of order $-k$ is
given by the following theorem:

\begin{theorem}
	\begin{equation*}
	y_{2}(n,k;\lambda )=\frac{2^{n}\lambda ^{-k}}{(2k)!}\sum_{j=0}^{k}\left( 
	\begin{array}{c}
	k \\ 
	j%
	\end{array}%
	\right) E_{n}^{(-k)}\left( \frac{k}{2};\lambda ^{2}\right) .
	\end{equation*}
\end{theorem}

\begin{proof}
	By using (\ref{C1}) and (\ref{Cad3}), we obtain the following functional
	equation:%
	\begin{equation*}
	F_{y_{2}}(t,k;\lambda )=\frac{\lambda ^{-k}}{(2k)!}\sum_{j=0}^{k}\left( 
	\begin{array}{c}
	k \\ 
	j%
	\end{array}%
	\right) F_{P1}\left( 2t,\frac{k}{2};j,\lambda ^{2}\right) .
	\end{equation*}%
	By using the above equation, we get%
	\begin{equation*}
	\sum_{n=0}^{\infty }y_{2}(n,k;\lambda )\frac{t^{n}}{n!}=\sum_{n=0}^{\infty }%
	\frac{2^{n}\lambda ^{-k}}{(2k)!}\sum_{j=0}^{k}\left( 
	\begin{array}{c}
	k \\ 
	j%
	\end{array}%
	\right) E_{n}^{(-k)}\left( \frac{k}{2};\lambda ^{2}\right) \frac{t^{n}}{n!}.
	\end{equation*}%
	Comparing the coefficients of $\frac{t^{n}}{n!}$ on both sides of the above
	equation, we arrive at the desired result.
\end{proof}

\begin{acknowledgement}
The paper was supported by the \textit{Scientific Research Project
Administration of Akdeniz University.}
\end{acknowledgement}

\end{document}